\documentclass[11pt]{amsart}
\usepackage{amsmath}
\usepackage{amssymb}
\usepackage{amsfonts}
\usepackage{mathrsfs}
\usepackage{graphicx}
\usepackage{epsfig}
\usepackage{enumerate}

\theoremstyle{plain}
\newtheorem{theorem}{Theorem}[section]
\newtheorem{proposition}{Proposition}[section]

\newtheorem{lemma}[proposition]{Lemma}

\theoremstyle{remark}
\newtheorem{remark}{Remark}[section]

 \theoremstyle{definition}
\newtheorem{definition}{Definition}[section]

\newtheoremstyle{citing}
  {3pt}
  {3pt}
  {\itshape}
  {}
  {\bfseries}
  {.}
  {.5em}
  {\thmnote{#3}}

\newcommand{\sA}{\mathscr{A}}

\newcommand{\cB}{\mathcal{B}}

\theoremstyle{citing}

\newcommand{\eps}{\varepsilon}

\newcommand{\hC}{\widehat{C}}

\newcommand{\Crit}{\textrm{Crit}}

\newcommand\diam{\mathop{{\rm diam}}}

\newcommand{\R}{\mathbb{R}}

\newcommand{\tf}{\hat{f}}
\newcommand{\tc}{\hat{c}}

\newcommand{\ser}{\textrm{SER}}

\newcommand{\PR}{\textrm{PR}}
\begin{document}

\title[Invariance of the Collet-Eckmann condition]{Topological invariance of the Collet-Eckmann condition for one-dimensional maps}

\author{Huaibin Li }
\address{School of Mathematics and Statistics, Henan University, Kaifeng 475004,
China}
\noindent\email{lihbmath@henu.edu.cn}

\thanks{{\it 2000 Mathematics Subject Classification}: 37C05, 37C15, 37F10}
\keywords{One-dimensional map, the Collet-Eckmann condition, the slow recurrence condition, topological invariance}
\thanks{The author was supported by the National Natural Science Foundation of China (Grant No. 11471098) and the China Scholarship Council (CSC No. 201508410033)  }

\maketitle

\begin{abstract}
 This paper is devoted to study the  topological invariance of several non-uniform hyperbolicity conditions of one-dimensional maps. In contrast with the case of maps with only one critical point, it is known that for maps with several critical points the Collet-Eckmann condition is not in itself invariance under topological conjugacy. We show that the Collet-Eckmann condition together with any of several slow recurrence conditions is invariant under topological conjugacy. This extends and gives a new proof of a result by Luzzatto and Wang that also applies to the complex setting.
\end{abstract}

\section{Introduction}\label{sec:intro}
Let $I$ be a compact interval of $\R.$ Recall that a non-injective continuous map $f: I\to I$ is \emph{multimodal}, if there is a finite
partition of $I$ into intervals on each of which $f$ is injective.
For a differentiable multimodal map $f: I\to I$,  a point of~$I$ is \emph{critical point of~$f$} if the derivative of~$f$ vanishes at it. We denote by~$\Crit(f)$ the set of critical points of~$f$.

\begin{definition}
We say that an interval map $f$ satisfies the \emph{Collet-Eckmann condition} (abbreviated CE condition) if all the periodic points of $f$ are hyperbolic
repelling, and if there are constants $C>0$ and $\lambda >1$ such that for each critical point $c$ of $f$, we have $$|Df^n(f(c))|\geq C\lambda^n.$$
\end{definition}
This type of condition was first introduced by P. Collet and J.-P. Eckmann in~\cite{ColEck83} providing a large class of $S$-unimodal maps of the interval having a finite absolutely continuous invariant measure.
Further investigations by T. Nowicki, S. van Strien, and D. Sands proved the
equivalence of those and several similar properties in the real $S$-unimodal
setting, see~\cite{Now85a, NowSan98} and see also~\cite{Prz98} for rational maps.

 Recall that two multimodal maps $f: I \to I$ and $g: J \to J$ are \emph{topologically conjugate}, if
there exists a homeomorphism $h: I\to J$ such that for each $x\in I$ we have $h(f(x)) = g(h(x))$.  Moreover, if both $h$ and $h^{-1}$ are  H\"{o}lder continuous, then we say that $f$ and $g$ are \emph{bi-H\"{o}lder conjugate}. A
topological conjugacy preserves topological properties of maps such as periodic orbits.

The problem of topological invariance of the Collet-Eckmann condition for $S$-unimodal maps was posed in~\cite{vSt88a} as well as by Guckenheimer and Misiurewicz in the early 1980s. In~\cite{NowPrz98} the authors proved that the Collet-Eckmann condition is invariant under topological conjugacy within
the class of $S$-unimodal maps.  This result does not however generalize to the multimodal setting, see for example~\cite{Mih0810, PrzRivSmi03} for some
counterexamples: a pair of topologically conjugate multimodal maps one of which satisfies Collet-Eckmann condition and the other does not.
But in~\cite{LuzWan06} the authors proved that a strengthened version of the Collet-Eckmann condition for multimodal maps is topologically invariant, see~\cite[Main Theorem]{LuzWan06}, and also~\cite{ wanglanyu01}, for a precise statement.  See also ~\cite{LiShe13} for a recent related result.

In this paper, we give a conceptual proof of~\cite[Main Theorem]{LuzWan06} that is also simpler and shorter than the proof  given by Luzzatto and Wang in~\cite{LuzWan06} that relies heavily on delicate combinatorial arguments. In fact, our method of proof allows us to prove variants of~\cite[Main Theorem]{LuzWan06} for different (and more natural) notions of slow recurrence. We illustrate this by stating Theorems~\ref{thm:1}, ~\ref{thm:2} and~\ref{thm:3}, but other variants are straightforward to obtain.  The main tool is~\cite[Proposition~5.2]{Riv1204} that states a conjugacy between two Lipschitz continuous multimodal
maps that satisfy the Exponential Shrinking of Components condition is bi-H\"{o}lder continuous. This allows us to show that for maps satisfying a slow recurrence condition, the Topological Collet-Eckmann condition implies the Collet-Eckmann condition.

Let us be more precise. Let $I$ be a compact interval of $\R.$
A $C^1$ multimodal map~$f : I \to I$ is \emph{of class~$C^3$ with non-flat critical points}, if
the map~$f$ is of class~$C^3$ outside $\Crit(f)$, and if
for each critical point~$c$ of~$f$ there exists a number $\ell_c>1$ and diffeomorphisms~$\phi$ and~$\psi$ of~$\R$ of class~$C^3$, such that $\phi(c)=\psi(f(c))=0$, and such that on a neighborhood of~$c$ on~$I$, we have
$|\psi\circ f| = |\phi|^{\ell_c}.$

In what follows, let $\sA_\R$ denote the class of all
$C^3$ interval maps with non-flat critical points and with all
periodic points hyperbolic repelling.

\subsection{Statement of results}
In this subsection, we state our main results. A multimodal map $f: I\to I$ in  $\sA_\R$ is \emph{topologically exact}, if for every open subset~$U$ of~$I$ there is an integer $n\geq 1$ such
that $f^n(U) = I.$
\begin{definition}
 Given $\beta$ in $(0, 1),$  we  say that a multimodal  map $f$ in  $\sA_\R$  satisfies the \emph{stretched exponential recurrence condition with respect to exponent $\beta$}
(abbreviated $f\in \ser(\beta)$), if there exists $C > 0$ such that for any two critical points $c, c'$ in $\Crit(f)$ and any positive integer
$n \geq 1$, we have
$$|f^n(c)- c'|\geq C\exp\left(-n^{\beta}\right).$$ Moreover, we say that the map $f$ satisfies the \emph{stretched exponential recurrence condition,} if $f\in \ser(\beta)$ for some $\beta$  in $(0, 1)$.
\end{definition}

\begin{theorem}\label{thm:1}
Let $f, \tf$ be two multimodal maps in $\sA_\R$ that are topologically exact, and that are topologically conjugate by a conjugacy preserving critical points. Assume that
$f$ satisfies both the Collet-Eckmann condition and the stretched exponential  recurrence condition. Then $\tf$ also  satisfies both the Collet-Eckmann condition and  the stretched exponential  recurrence condition.
\end{theorem}

\begin{definition}
Given $\beta>0$, we  say that a multimodal  map $f$ in  $\sA_\R$  satisfies the \emph{exponential recurrence condition with respect to $\beta$}, if there exists $C > 0$ such that for any two critical points $c, c'$ in $\Crit(f)$ and any positive integer
$n \geq 1$, we have
$$|f^n(c)- c'|\geq C\exp\left(-{\beta}n\right).$$ Moreover, we say that $f$ in  $\sA_\R$  satisfies the \emph{subexponential recurrence condition} if for every $\beta>0$ the map $f$ satisfies the exponential recurrence condition with respect to $\beta$.
\end{definition}

\begin{theorem}\label{thm:2}
Let $f, \tf$ be two multimodal maps in $\sA_\R$ that are topologically exact, and that are topologically conjugate by a conjugacy preserving critical points. Assume that
$f$ satisfies both the Collet-Eckmann condition and the subexponential  recurrence condition. Then~$\tf$ also  satisfies both the Collet-Eckmann condition and the subexponential  recurrence condition.
\end{theorem}

\begin{definition}
 Given $\beta>0,$  we  say that a multimodal  map $f$ in  $\sA_\R$  satisfies the \emph{polynomial recurrence condition with respect to exponent $\beta$}
(abbreviated $f\in \PR(\beta)$), if there exists $C > 0$ such that for any two points $c, c'\in \Crit(f)$ and any
$n \geq 1$, we have
$$|f^n(c)- c'|\geq Cn^{-\beta}.$$ Moreover, $f$ is said to satisfy the \emph{polynomial recurrence condition } if $f\in \PR(\beta)$ for some $\beta>0$.
\end{definition}
\begin{theorem}\label{thm:3}
Let $f, \tf$ be two multimodal maps in $\sA_\R$ that are topologically exact, and that are topologically conjugate by a conjugacy preserving critical points. Assume that
$f$ satisfies both the Collet-Eckmann condition and the polynomial recurrence condition. Then $\tf$ also  satisfies both the Collet-Eckmann condition and polynomial recurrence condition.
\end{theorem}

Recall that $f$ satisfies \emph{the slow recurrence condition} of~\cite{LuzWan06} if every critical point $c$ of $f$ satisfies the following property:
$$\lim_{\delta\to 0^+}\liminf_{n\to +\infty}\frac{1}{n}\sum_{1\leq i \leq n, \atop f^i(c)\in \cB_f(\delta)}\log d(f^i(c))=0,$$ where
$ \cB_f(\delta):=B(\Crit(f), \,\, \delta)$ and $d(x) : = \min\{{|x-c|: c \in \Crit(f)}\}.$

The following is~\cite[Main Theorem]{LuzWan06}, and we give a new proof of it in Section~\ref{sec:LuzWangTheorem}.
\begin{theorem}\label{thm:lw06}
Let $f$ be a multimodal maps in $\sA_\R$. Suppose that $f$ satisfies both the Collet-Eckmann condition and the slow recurrence condition. Then every map $g$ in $\sA_\R$  that is topologically
conjugate to $f$ by a conjugacy preserving critical points also satisfies both  Collet-Eckmann condition and slow recurrence condition.
\end{theorem}
\begin{remark}
The strategy used by Luzzatto and Wang to prove Theorem~\ref{thm:lw06} is to define a new condition
which they call the {\it Topological Slow Recurrence (TSR) condition}, see the precise definition on page $349$ in~\cite{LuzWan06},
which depends only on the combinatorics of the critical orbits. In particular TSR is invariant under topological conjugacy. Then they proved that this condition is equivalent to the simultaneous occurrence of the slow recurrence condition and the Collet-Eckmann conditions.
\end{remark}
\subsection{Organization}
The paper is organized as follows. In Section~\ref{sec:preliminaries}, we recall some definitions of non-uniform hyperbolicity conditions, and collect
some results of non-uniform hyperbolicity conditions and Koebe distortion. In Section~\ref{sec:proofsoftheorems}, we give the proofs of our
Theorems~\ref{thm:1}, \ref{thm:2} and~\ref{thm:3}. In Section~\ref{sec:LuzWangTheorem}, we give a new proof of Theorem~\ref{thm:lw06}. In Appendex~\ref{sec:rational},  we give an analog  theorem in the complex setting, see Theorem~\ref{thm:rational}.

\subsection{Acknowledgments}
\label{ss:Acknowledgements}
The author would like to thank Juan Rivera-Letelier  for his stimulating conversations and discussions, and
for his useful comments and corrections to earlier versions of this paper. The author also thanks Weixiao Shen and Mike Todd for some helpful
comments in the early versions of this manuscript.
Finally, the author would  also like to thank the School of  Mathematics and Statistics in University of St Andrews for
the optimal working conditions provided to his visiting, where this article was completed.

\section{Preliminaries}\label{sec:preliminaries}
In what follows, let $I$ and  $J$ be two compact intervals of $\R$. In this section, we collect some known results which will be used in proving our theorems.

\subsection{Topological Collect-Eckmann condition}  Let $f: I\to I$
be a multimodal  map and fix $r > 0.$ Recall that given
an integer $n \geq 1,$ the {\it criticality} of $f^n$ at a point $x$ of $I$ with respect to $r$ is the number of those $j$
in $\{0, \cdots, n-1\}$ such that the connected component of $f^{(n-j)}(B(f^n(x), r))$ containing
$f^j(x)$ contains a critical point of~$f$ in $\Crit(f).$  We say that  $f$ satisfies the \emph{Topological Collet-Eckmann  condition} (abbreviated TCE condition), if for some choice of $r > 0$ there are constants $D\geq 1$
and~$\theta$ in $(0, 1),$ such that the following property holds: For each point~$x$ in~$I$ the
set $G_x$ of all those integers $m\geq 1$ for which the criticality of $f^m$ at $x$ is less than
or equal to $D,$ satisfies $$\liminf_{n\to +\infty}\frac{1}{n} \#(G_x \cap \{1, \cdots, n\}) \geq \theta.$$  The TCE condition was first introduced in~\cite{NowPrz98}.  Clearly, the TCE condition is topologically invariant.

Recall that an interval map $f:I \to I$ satisfies the \emph{Exponential Shrinking of Components condition with respect to $\lambda > 1$}, if there are constants $C>1$ and $\delta > 0$
such that for every interval $J$ contained in $I$ that satisfies $|J|\leq \delta,$
the following holds: For every positive integer $n$ and every connected component
$W$ of $f^{-n}(J)$ we have the following inequality $$|W|\leq C\lambda^{-n}.$$ Moreover, we say that $f$ satisfies the \emph{Exponential Shrinking of Components condition} if there exists $\lambda>1$ such that $f$ satisfies the Exponential Shrinking of Components condition with respect to $\lambda$.

We will use the following fact that was proved by Rivera-Letelier in ~\cite{Riv1204}.
\begin{lemma}[Corollary~A and~C, \cite{Riv1204}]\label{tce1esc}
Let $f: I\to I$ be a multimodal  interval map in $\sA_\R$ that is topologically exact,
then the TCE condition is equivalent to the Exponential Shrinking of Components condition. Moreover, if $f$ satisfies the CE condition, then $f$ also satisfies the Exponential Shrinking of Components condition, and the TCE condition.
\end{lemma}

We will also use the following lemma.

\begin{lemma}[Proposition 5.2, \cite{Riv1204}]\label{conjugacybiholder}
Let $I$ and $J$ be two compact intervals of~$\R$. Let $f: I \to I$ be a Lipschitz continuous multimodal map
and  $g: J \to J$ a multimodal map satisfying the Exponential Shrinking of Components
condition. If $h: I \to J$ is a homeomorphism conjugating $f$ to $g$, then $h$ is H\"{o}lder continuous.
\end{lemma}

Given a multimodal map $f: I\to I,$ an integer $n\geq 1$ and a subset $J$ of $I$, a connected component of $f^{-n}(J)$ will be called a \emph{pull-back} of $J$ by $f^n.$ The following general facts of multimodal maps will be used for several times in what follows, see for example~\cite{Riv1206} for a proof.
\begin{lemma}[Lemma~A.2, \cite{Riv1206}]\label{l:pullstable}
Let $f: I \to I$ be an interval map in~$\sA$ that is topologically exact.
Then for every $\kappa> 0$ there is $\delta> 0$ such that for every~$x$ in~$I$, every integer~$n\geq 1$, and every pull-back $W$ of $B(x, \delta)$ by $f^n$, we have $|W| < \kappa.$
\end{lemma}

\subsection{Distortion Lemmas}
Let~$\tau > 0$ and let $J_1 \subset J_2$ be two intervals of $I$. We say that $J_1$ is \emph{$\tau$-well inside~$J_2$,} if both components of $J_2\setminus J_1$ have length at least~$\tau|J|$.

\begin{lemma}[Theorem~A, \cite{LiShe10a}]\label{koebe}
 Let $f:I \to I$ be a multimodal map in $\sA$. Then for each $\tau > 0$ there exist
$C>1$ and $\xi>0$ satisfying the following. Let $T\subset I$ be an open
interval, $J$ a closed subinterval of $T$  and an integer $s\geq 1$ such that the following hold:
\begin{enumerate}[1.]
\item $f^s: T \to f^s(T)$ is a diffeomorphism;
\item $|f^s(T)|\leq \xi$;
\item $f^s(J)$ is $\tau$-well inside $f^s(T).$
\end{enumerate}
Then for each pair $x$ and $y$ of points in $J$ we have
$$\frac{|Df^s(x)|}{|Df^s(y)|}\leq C.$$
Furthermore, $C \to 1$  as $\tau \to +\infty.$
\end{lemma}

\section{Proof of main Theorems}\label{sec:proofsoftheorems}
This section is devoted to provide the proofs of  Theorems~\ref{thm:1}, ~\ref{thm:2} and~\ref{thm:3}, which depend on the following two propositions.
\begin{proposition}\label{Pro:both}
Let $f:I\to I$ and  $\tf:J\to J$ be  multimodal maps in $\sA_\R$ that are topologically exact, and that are topologically conjugate by a conjugacy preserving critical points.  Assume that
$f$ satisfies both the Exponential Shrinking of Components condition and $\ser(\beta)$ for some $\beta \in (0,\, 1)$. Then~$\tf$ satisfies the TCE condition, and there is $\beta'>0$ such that $\tf \in \ser( \beta').$
\end{proposition}
\begin{proof}
 By Lemma~\ref{tce1esc} and the topological invariance of the TCE condition, we only need to prove that the map~$f$  satisfies  the stretched exponential recurrence condition. Combing Lemma~\ref{tce1esc} and Lemma~\ref{conjugacybiholder}, we know that there is a bi-H\"{o}lder continuous $h:I \to J$ such that for each $x\in I$ we have $h(f(x))=\tf(h(x))$. Let $K>0$ and $\alpha\in (0,1]$ be the constants such that for each $x,y\in J$ we have $$|h^{-1}(x)-h^{-1}(x)|\leq K|x-y|^\alpha.$$ Let $C>0$ such that for any two points $c, c'\in \Crit(f)$ and any integer
$n \geq 1$, we have
$$|f^n(c)- c'|\geq C\exp(-n^{\beta}).$$  Let $\tc$ and $\tc'$ be any two critical points of $\tf$, and $n\geq 1$. Notice that  $h^{-1}(\tc)$ and  $ h^{-1}(\tc')$  are critical points of $f$. It follows that
\begin{multline*}
K|\tf^n(\tc)-\tc'|^\alpha\geq |h^{-1}(\tf^n(\tc))-h^{-1}(\tc')|\\=|f^n(h^{-1}(\tc))-h^{-1}(\tc')|\geq C \exp\left(-n^{\beta}\right).
\end{multline*}
This gives us $$|\tf^n(\tc)-\tc'|\geq \left(\frac{C}{K}\right)^{1/\alpha}\exp\left(-\frac{n^{\beta}}{\alpha}\right).$$  Therefore,  $\tf$  satisfies  $\ser( \beta')$ for some  $\beta'>\beta$ depending only on  $\alpha$ and $\beta,$ and the proof is complete.
\end{proof}

Following the proof of Proposition~\ref{Pro:both}, we can obtain the following two lemmas. The details are left to the interested reader to check.
\begin{lemma}\label{lem:erc}
 Let $f:I\to I$ and  $\tf:J\to J$ be  multimodal maps in $\sA_\R$ that are topologically exact, and that are topologically conjugate by a conjugacy preserving critical points.  Assume that
$f$ satisfies both the Exponential Shrinking of Components condition and the subexponential recurrence condition. Then~$\tf$ satisfies both of the TCE condition and the subexponential recurrence condition.
\end{lemma}

\begin{lemma}\label{lem:prc}
 Let $f:I\to I$ and  $\tf:J\to J$ be  multimodal maps in $\sA_\R$ that are topologically exact, and that are topologically conjugate by a conjugacy preserving critical points.  Assume that
$f$ satisfies both the Exponential Shrinking of Components condition and the polynomial recurrence condition of $\beta>0$. Then~$\tf$ satisfies the TCE condition, and there is $\beta'>0$ such that $\tf\in \PR(\beta').$
\end{lemma}

\begin{proposition}\label{pro:ce}
Let $f:I \to I$ be an interval map in $\sA_\R$ that satisfies the Exponential Shrinking of Components condition with respect to some $\lambda>1$, then there is $\beta_0>0$ such that the following holds. If $f$ satisfies the exponential recurrence condition with respect to $\alpha$ in $(0, \beta_0)$,  then $f$ satisfies the Collet-Eckmann condition.
\end{proposition}

\begin{proof}
Putting $M:=\sup_I|Df|,$  by the Exponential Shrinking of Components condition we have that $M>1$. Moreover, reducing $\lambda$ if necessary, we assume that $\lambda$ is in $(1, M)$, and let $\delta_0>0$ be given by the Exponential Shrinking of Components condition with respect to $\lambda$. Setting $$\beta_0=\frac{(\log \lambda)^2}{2(\log M- \log \lambda)},$$ we prove the proposition with $\beta_0.$ To do this, fix $\alpha$ in $(0, \beta_0),$ and assume that $f$  satisfies the exponential recurrence condition with respect to $\alpha.$ Then there exists $C_\alpha > 0$ such that for any two critical points $c, c'$ in $\Crit(f)$ and any
$n \geq 1$, we have
\begin{equation}\label{e:erc}
|f^n(c)- c'|\geq C_\alpha\exp\left(-\alpha n\right).
 \end{equation}
 For any critical value $v$ of $f$ and every positive integer~$n,$  let $m\geq 0$ be the smallest integer such that $$\lambda^{-m}\leq C_\alpha \exp\left(- \alpha (n+1)\right).$$  Since $f^m(B(f^n(v), \delta_0 M^{-m}))\subset B(f^{n+m}(v), \delta_0),$ the Exponential Shrinking of Components condition implies that for each $j$ in $\{0, 1, \cdots, n\},$ the connected component~$W_j$ of $f^{-(n-j)}(B(f^n(v),\delta_0 M^{-m}))$ containing $f^{j}(v)$ satisfies $$\diam (W_j)\leq \lambda^{-(m+n-j)}\leq \lambda^{-(n-j)} C_\alpha \exp\left(-\alpha (n+1)\right)< C_\alpha \exp\left(- \alpha j\right).$$ In particular, $\diam (W_0)\leq \lambda^{-(m+n)}.$  On the other hand, combining inequality~(\ref{e:erc}), we have that each $W_j$ is disjoint from $\Crit(f)$. It follows that the map $$f^n: W_0\to B(f^n(v), \delta_0 M^{-m})$$ is a diffeomorphism. By Lemma~\ref{koebe},  there exists some constant $C>0$ independent of~$n$ and~$v$, such that the following holds
\begin{equation}\label{d1}
|Df^n(v)|\geq \frac{CM^{-m}}{\diam (W_0)}\geq \frac{C M^{-m}}{\lambda^{-(m+n)}} = C \lambda ^{n+m} M^{-m}.
\end{equation}
If $m=0$, then we obtain $$|Df^n(v)|\geq C\lambda^n.$$ Otherwise, setting $$C':=(\lambda^{-1} C_\alpha )^{\frac{\log M}{\log \lambda} -1}$$ and by the minimality of $m$ we have
\begin{align*}
\lambda^m M^{-m}&=(\lambda^{-m})^{\frac{\log M}{\log \lambda} -1}>\left(\lambda^{-1} C_\alpha\exp\left(- \alpha (n+1)\right)\right)^{\frac{\log M}{\log \lambda} -1}\\&= C'\exp\left(- \alpha (n+1) \left(\frac{\log M}{\log \lambda }-1\right)\right).
\end{align*}
If putting $\alpha_*:=2\alpha \left(\frac{\log M}{\log \lambda }-1\right)$, then we have $\lambda_1:=\lambda \exp(-\alpha_*)>1$ by the definition of $\alpha$. Moreover,  putting $C_0:=C C'>0$, and using with (\ref{d1}), for every positive integer $n,$ we have
\begin{align*}
|Df^n(v)|&\geq C C'\lambda^n \exp\left(- \alpha (n+1) \left(\frac{\log M}{\log \lambda }-1\right)\right)\\&\ge C_0 \lambda^n \exp\left(-\alpha_* n\right)=C_0 \lambda_1^n .
\end{align*}
This implies that the map $f$ satisfies the Collet-Eckmann condition for some~$\lambda_1$, and completes the proof.
\end{proof}

\begin{proof}[Proof of Theorem~\ref{thm:1}] By Lemma~\ref{tce1esc}, we have that both of $f$ and $\tf$ satisfy the Exponential Shrinking of Components condition. It follows from  Proposition~\ref{Pro:both} that $\tf $ satisfies the stretched exponential  recurrence condition. Therefore, we have that $\tf $ satisfies the subexponential  recurrence condition. Moreover, by Proposition~\ref{pro:ce}, we have that $\tf$ satisfies the Collet-Eckmann condition, and complete the proof.
\end{proof}
\begin{proof}[Proof of Theorem~\ref{thm:2}] By Lemma~\ref{tce1esc}, we have that both $f$ and $\tf$ satisfy the Exponential Shrinking of Components condition. It follows from  Lemma~\ref{lem:erc} that $\tf $ satisfies the subexponential  recurrence condition.  Moreover, by Proposition~\ref{pro:ce}, we have that~$\tf$ satisfies the Collet-Eckmann condition, and complete the proof.
\end{proof}
\begin{proof}[Proof of Theorem~\ref{thm:3}] By Lemma~\ref{tce1esc}, we have that both $f$ and $\tf$ satisfy the Exponential Shrinking of Components condition. It follows from  Lemma~\ref{lem:prc} that $\tf $ satisfies the polynomial recurrence condition. Therefore, we have that~$\tf $ satisfies the subexponential  recurrence condition.  Moreover, by Proposition~\ref{pro:ce}, we have that~$\tf$ satisfies the Collet-Eckmann condition, and complete the proof.
\end{proof}

\section{Proof of Theorem~\ref{thm:lw06}}\label{sec:LuzWangTheorem}
In this section, we  give a new and conceptual proof of Theorem~\ref{thm:lw06}.
In what follows, let~$f: I\to I$ be a multimodal map in $\sA_\R$, and suppose that $f$ satisfies both the Collet-Eckmann condition and the slow recurrence condition. Let $g: J\to J$ be a multimodal map in $\sA_\R$ that is
 topologically conjugate to $f$ by the conjugacy $h:I\to J$ preserving critical points. By Lemmas~\ref{tce1esc} and~\ref{conjugacybiholder},  we know that $h$ is bi-H\"{o}lder continuous. It follows that there exist two constants $K>1$ and $\alpha$ in $(0,1)$ such that for every $x,\, y$ in $I$ we have $$|h(x)-h(y)|\leq K|x-y|^\alpha;$$ and for every $s,\, t$ in $J$ we have $$|h^{-1}(s)-h^{-1}(t)|\leq K|s-t|^\alpha.$$

First, we prove the following proposition before proving the theorem.
\begin{proposition}\label{pro:slowyexp2ce}
Assume that $g$ in $\sA_\R$ satisfies simultaneously the Exponential Shrinking of Components condition and the slow recurrence condition. Then $g$ satisfies the Collet-Eckmann condition.
\end{proposition}
\begin{proof}
 By the hypothesis that all critical points of~$g$ are non-flat, there exist $\ell>1$, $\kappa_0>0$ and $L>1$ such that for every critical point~$c^g$ in $\Crit(g)$ and every $x\in B(c^g, \kappa_0),$ we have
\begin{equation}\label{e:nonflat}
|Df(x)|\geq \frac{1}{L}|x-c^g|^{\ell}.
\end{equation}
Let $\eta_1$ be the constant gave by Lemma~\ref{l:pullstable} for $\kappa=\kappa_0,$ and let $C, \xi$ be the constants given by  Lemma~\ref{koebe} with $\tau =\frac{1}{2}$. Moreover, since~$g$ satisfies  the Exponential Shrinking of Components condition, then there are constants $\hC>2$, $\eta_2 > 0$
and $\lambda > 1$ such that for every interval $T$ contained in $J$ that satisfies $|T|\leq 2\eta_2,$
the following holds: For every positive integer $n$ and every connected component
$W$ of $g^{-n}(J)$ we have $|W|\leq \hC\lambda^{-n}.$

 Now setting $\eta_0=\min\{\eta_1,\, \eta_2, \, \xi\},$ and fix a critical point $c^g$ in $\Crit(g)$ and put $v^g:=g(c^g).$ Put $$\vartheta:=\frac{\log(2(C\hC)^{-1})}{\log \eta_0}+\log(\lambda L)+ 2\ell\,\, \,\,\text{and}\,\,\,\,\, \eps_*:=\frac{\log\lambda}{2\vartheta}.$$ Since the map~$g$ satisfies the slow recurrence condition, then there exist $\eta_4$ in $(0,\, \min\{\eta_0, \, \kappa_0\})$ and a positive integer $N_1$ such that for every $\delta$ in $(0, \eta_4]$ and every positive integer~$n\geq N_1,$ the following inequality holds. \begin{equation}\label{e:gss}
\frac{1}{n}\sum_{\substack{1\leq i\leq n\\ g^i(c^g)\in \cB_g(\delta)}}-\log d(g^i(c^g))< \eps_*.
\end{equation}
In particular,
 \begin{equation}\label{e:nnnumbercri}
 \#\{1\leq i \leq n: g^i(c^g)\in \cB_g(\delta)\}\leq \frac{n\eps_*}{-\log \eta_4}< n\eps_*.
 \end{equation}

 For every $n\geq N_1$,  let us define a {\em quasi-chain}
$\{\widehat W_k\}_{k=0}^n$ by the following rules:
\begin{itemize}
\item[(i)] $\widehat
W_n=B(g^n(v^g), \eta_4)$;
\item[(ii)] Once $\widehat W_{k+1}\ni g^{k+1}(v^g)$
is defined, letting $\widehat W_k'$ be the connected component of
$g^{-1}( \widehat W_{k+1})$ which contains $g^k(v^g)$;
\item[(iii)] If
$\widehat W_k'$ contains no critical point of $g$, then $\widehat
W_k=\widehat W_k'$, and otherwise, let $$\widehat W_k=B(g^k(v^g),
\eta_4).$$
\end{itemize}

Let $n_0=n$ and let $n_1>n_2>\cdots> n_m$ be all integers in
$\{1,\ldots, n-1\}$ such that $\widehat W_{n_i}$ contains a
critical point of $g$. Then by inequality~(\ref{e:nnnumbercri}) we have
\begin{equation}\label{bothside}
m \leq \frac{n\eps_*}{-\log \eta_4}\leq n\eps_*.
\end{equation}
 Moreover, for every $1\leq i\leq m$, the map $g^{n_{i-1}-n_i-1}: \widehat W_{n_i+1}\to
\widehat W_{n_{i-1}}$ is a diffeomorphism,
$\widehat W_{n_{i-1}}\subset B(c, 2\eta_0)$ for some $c\in\Crit(g)$ and $g^{n-n_{i-1}}(v^g)$ is the middle point of~$W_{n_{i-1}}$. Then by Lemma~\ref{koebe}, and the Exponential Shrinking of Components condition, we have
\begin{equation}\label{e:chainde1}
|Dg^{n_{i-1}-n_i-1}(g^{n-n_i+1}(v^g))|\ge C_1\lambda^{n_{i-1}-n_i-1}.
\end{equation}
where $C_1=2\eta_4(C \hC)^{-1} \in (0,1).$  Similarly, we have
\begin{equation}\label{e:chainde2}
|Dg^{n_{m}}((v^g))|\ge C_1\lambda^{n_{m}}.
\end{equation}

On the other hand, by the choice of $\eta_4$ and inequality~(\ref{e:nonflat}), for every positive integer~$i$ such that $1\leq i\leq m$ we have
\begin{equation}\label{e:chainde3}
|Dg(g^{n_{i}}(v^g))|\ge \frac{1}{L}|g^{n_{i}}(v^g)-c^g|^{\ell} \geq \frac{(d(g^{n_{i}}(v^g)))^{\ell}}{L}.
\end{equation}
Therefore, combining inequalities~(\ref{e:chainde1}), (\ref{e:chainde2}) and~(\ref{e:chainde3}), we have
\begin{align*}
|Dg^n(v^g)|&=\prod_{i=1}^m|Dg^{n_{i-1}-n_i-1}(g^{n-n_i+1}(v^g))|\cdot\prod_{i=1}^m|Dg(g^{n_{i}}(v^g))|\cdot|Dg^{n_{m}}((v^g))|\\&\geq \lambda^{n}C_1^m\left(\frac{ 1}{\lambda L}\right)^{m}\cdot \prod_{i=1}^m(d(g^{n_{i}}(v^g)))^{\ell}.
\end{align*}
It follows from~(\ref{bothside}) and the slow recurrence condition that
\begin{multline*}
\frac{\log|Dg^n(v^g)|}{n}\geq \log\lambda+\frac{m}{n} \log C_1-\frac{m}{n}\log(\lambda L)+ \frac{\ell}{n}\sum_{i=1}^m\log d(g^{n_{i}}(v^g))\\\geq\log\lambda-\frac{n\eps_*}{\log \eta_4}\frac{1}{n} \log \left(\frac{2\eta_4}{C\hC}\right)-\frac{n\eps_*}{n}\log(\lambda L)+ \frac{\ell}{n}\sum_{i=1}^m\log d(g^{n_{i}}(v^g))\\=\log\lambda-\frac{\eps_*}{\log \eta_4} \log \left(\frac{2\eta_4}{C\hC}\right)-\eps_*\log(\lambda L)+ \frac{\ell}{n}\sum_{i=1}^m\log d(g^{n_{i}}(v^g))\\\geq \log\lambda-\eps_*\left(\frac{\log(2(C\hC)^{-1})}{\log \eta_4}+\log(\lambda L)+ 1\right)+ \frac{\ell}{n}\sum_{\substack{1\leq i\leq n\\ g^i(c^g)\in \cB_g(\eta_4)}} \log d(g^{n_{i}}(v^g))\\\geq \log \lambda- \eps_*\left(\frac{\log(2(C\hC)^{-1})}{\log \eta_4}+\log(\lambda L)+ 1+\ell\right)
\end{multline*}
Therefore, by the slow recurrence condition and the definition of $\eps_*$, we have
\begin{align*}
 \liminf_{n\to \infty}\frac{1}{n}\sum_{i=1}^{n}\log |Dg^i(v^g)|&\geq \log \lambda -\eps_*\left(\frac{\log(2(C\hC)^{-1})}{\log \eta_0}+\log(\lambda L)+ 1+\ell\right)\\&= \frac{1}{2} \log \lambda,
\end{align*}
 and complete the proof of the proposition.
\end{proof}

\begin{proof}[Proof of Theorem~\ref{thm:lw06}]
Let $h$, $K$ and $\alpha$ be as given in the beginning of this section. First, combining Lemma~\ref{tce1esc} and the topological invariance of the TCE condition, we know that $g$ satisfies the Exponential Shrinking of Components condition. In view of Proposition~\ref{pro:slowyexp2ce}, it remains to prove that the map~$g$ satisfies the slow recurrence condition. In fact, it suffices to prove that for every critical point $c^g$ in $\Crit(g)$ and  any $\eps>0,$ there exist $\delta_0>0$ and a positive integer~$N$ such that for every $\delta$ in $(0, \delta_0]$ and every positive integer~$n\geq N,$ the following inequality holds. \begin{equation}\label{e:gs}
\frac{1}{n}\sum_{\substack{1\leq i\leq n\\ g^i(c^g)\in \cB_g(\delta)}}-\log d(g^i(c^g))< \eps.
 \end{equation}
 Fix a critical point~$c^g$ in $\Crit(g)$ and $\eps>0.$ Putting $c^f:=h^{-1}(c^g),$ then we have by the hypothesis that $c^f$ is in $\Crit(f)$. Moreover, by the hypothesis on $f$ we have that for $\eps':=\frac{\alpha\eps}{1+\log K}>0,$ there exist $\delta_1$ in $(0, e^{-1})$ and a positive integer $N$ such that for every $\delta$ in $(0, \delta_1]$ and every positive integer~$n\geq N,$ the following inequality holds.
 \begin{equation}\label{e:fs}
 \frac{1}{n}\sum_{\substack{1\leq i\leq n\\ f^i(c^f)\in \cB_f(\delta)}}-\log d(f^i(c^f))< \eps'.
 \end{equation}
 It follows that
 \begin{equation}\label{e:numbercri}
 \#\{1\leq i \leq n: f^i(c^f)\in \cB_f(\delta_1)\}\leq \frac{n\eps'}{-\log \delta_1}< n\eps'.
 \end{equation}
 Now setting $\delta_0:=\left(\frac{\delta_1}{K}\right)^{\frac{1}{\alpha}}$, and note that for every  point $c_1^g$ in $\Crit(g)$ and every  positive integer $i$ we have that  $h^{-1}(c_1^g)$ is in $\Crit(f)$, and that
 \begin{equation*}
 |f^i(c^f)-h^{-1}(c_1^g)|=|h^{-1}(g^i((c^g))-h^{-1}(c_1^g)|\leq K|g^i(c^g)-c_1^g|^\alpha.
 \end{equation*}
 It follows that for every $\delta$ in $(0, \delta_0]$,   every positive integer $i$ such that $g^i(c^g)\in \cB_g(\delta),$ we have that $f^i(c^f)$ is in $\cB_f(\delta_1)$ and $$d(g^i(c^g))\geq \left(\frac{d(f^i(c^f))}{K}\right)^{\frac{1}{\alpha}}.$$ Therefore, combining inequalities~(\ref{e:fs}) and~(\ref{e:numbercri}), for every $\delta$ in $(0, \delta_0]$ and every positive integer $n\geq N,$ we have
 \begin{multline*}
 \frac{1}{n}\sum_{\substack{1\leq i\leq n\\ g^i(c^g)\in \cB_g(\delta)}}-\log d(g^i(c^g))
 \leq \frac{1}{n}\sum_{\substack{1\leq i\leq n\\ f^i(c^f)\in \cB_f(\delta_1)}}-\log \left[\left(\frac{d(f^i(c^f))}{K}\right)^{\frac{1}{\alpha}}\right]
 \\ =\frac{1}{n}\sum_{\substack{1\leq i\leq n\\ f^i(c^f)\in \cB_f(\delta_1)}}-\frac{1}{\alpha}\log d(f^i(c^f))+\frac{1}{n}\sum_{\substack{1\leq i\leq n\\ f^i(c^f)\in \cB_f(\delta_1)}}\frac{\log K}{\alpha}
 \\ \leq \frac{1}{\alpha}\eps'+\frac{1}{n}\cdot n\eps'\cdot \frac{\log K}{\alpha}=\frac{1+\log K}{\alpha}\cdot \eps'=\eps.
 \end{multline*}
 This proves inequality~(\ref{e:gs}), and so the map $g$ satisfies the slow recurrence condition. The proof of the theorem is completed.
\end{proof}

\appendix
\section{Rational maps}~\label{sec:rational}
In this appendix, let $f$ be a rational map of degree at least $2$. Let $\Crit( f )$ denote the set of critical points of $f$, let $J( f )$ denote the Julia set of  $f$ and let $$\Crit'(f)=\Crit( f )  \cap J( f ).$$  As the definitions  defined in the real setting, we can also define the Collect-Eckmann condition,  the stretched exponential condition,  the TCE condition and the Exponential Shrinking of Components condition for rational maps.

More precisely, we say that  the rational map   $f$ satisfies the \emph{Collet-Eckmann condition} (abbreviated CE condition) if all the periodic points of $f$ are hyperbolic
repelling, and if there are constants $C>0$ and $\lambda >1$ such that for each critical point $c$ in  $\Crit'(f)$, we have $$|Df^n(f(c))|\geq C\lambda^n.$$

Given $\beta$ in $(0, 1),$  we say that the rational map $f$ satisfies the \emph{stretched exponential recurrence condition of exponent $\beta$}, if there exists $C > 0$ such that for any two critical points $c, c'$ in $\Crit'(f)$ and any
$n \geq 1$, we have
$$|f^n(c)- c'|\geq C\exp\left(-n^{\beta}\right).$$  Moreover, the map $f$ is said to satisfy the \emph{stretched exponential recurrence condition} if~$f\in \ser(\beta)$ for some $\beta$  in $(0, 1)$.

Given $\beta>0$, we  say that a multimodal  map $f$ in  $\sA_\R$  satisfies the \emph{exponential recurrence condition with respect to $\beta$}, if there exists $C > 0$ such that for any two critical points $c, c'$ in $\Crit(f)$ and any positive integer
$n \geq 1$, we have
$$|f^n(c)- c'|\geq C\exp\left(-{\beta}n\right).$$ Moreover, we say that $f$ in  $\sA_\R$  satisfies the \emph{subexponential recurrence condition} if for every $\beta>0$ the map $f$ satisfies the exponential recurrence condition with respect to $\beta$.

 Given $\beta>0,$  we  say that a multimodal  map $f$ in  $\sA_\R$  satisfies the \emph{polynomial recurrence condition with respect to exponent $\beta$}, if there exists $C > 0$ such that for any two points $c, c'\in \Crit(f)$ and any
$n \geq 1$, we have
$$|f^n(c)- c'|\geq Cn^{-\beta}.$$ Moreover, $f$ is said to satisfy the \emph{polynomial recurrence condition } if $f\in \PR(\beta)$ for some $\beta>0$.

 We can also  use the analogy proof of Theorem~\ref{thm:1}  to prove the following analog of Theorem~\ref{thm:1}  in the complex setting. In fact, this is a feature of the method of the proof of Theorem~\ref{thm:1}.

\begin{theorem}\label{thm:rational}
Let $f, \tf$ be two rational maps of degree at least two, and that are topologically conjugate by a conjugacy preserving critical points. Assume that
$f$ satisfies both the Collet-Eckmann condition and the stretched exponential (resp. subexponential, polynomial) recurrence condition. Then $\tf$ also  satisfies both the Collet-Eckmann condition and  the stretched exponential (resp. subexponential, polynomial)  recurrence condition.
\end{theorem}

\begin{proof}
The proof of this theorem follows the same outline as that of Theorem~\ref{thm:1}, replacing Lemma~\ref{tce1esc} and Lemma~\ref{conjugacybiholder}  by the
following two results, respectively:
\begin{itemize}
\item  TCE condition is  equivalent to the Exponential Shrinking of Components condition for rational maps,
see~\cite[Main Theorem]{PrzRivSmi03};
\item
The conjugacy between two rational maps of degree at least~$2$ is bi-H\"{o}lder continuous, see~\cite[Theorem~A]{PrzRoh99}.
\end{itemize}
 The details are left to the reader.
\end{proof}

\end{document}